%% file: main.tex
\documentclass[11pt]{article}
\usepackage{glauber}
\usepackage{physics}
\usepackage{mathtools}

\usepackage{quasi}
\allowdisplaybreaks

\begin{document}

\title{Eigenvalue Bounds for Random Matrices via Zerofreeness}
\author{
    Sidhanth Mohanty\thanks{Northwestern University. Email: \texttt{sidhanth.mohanty@northwestern.edu}.
    Much of this work was done while the author was a postdoctoral researcher at MIT supported by NSF Award DMS-2022448.} \and
    Amit Rajaraman\thanks{MIT. Email: \texttt{amit\_r@mit.edu}. Supported by a MathWorks Fellowship. }
}
\date{\today}
\maketitle

\begin{abstract}
    We introduce a new technique to prove bounds for the spectral radius of a random matrix, based on using Jensen's formula to establish the zerofreeness of the associated characteristic polynomial in a region of the complex plane. Our techniques are entirely non-asymptotic, and we instantiate it in three settings:
    \begin{enumerate}[label=(\roman*)]
        \item The spectral radius of \emph{non-asymptotic Girko matrices}---these are asymmetric matrices $\bM \in \C^{n \times n}$ whose entries are independent and satisfy $\E \bM_{ij} = 0$ and $\E |\bM_{ij}^2| \le \frac{1}{n}$.
        \item The spectral radius of \emph{non-asymptotic Wigner matrices}---these are symmetric matrices $\bM \in \C^{n \times n}$ whose entries above the diagonal are independent and satisfy $\E \bM_{ij} = 0$, $\E |\bM_{ij}^2| \le \frac{1}{n}$, and $\E |\bM_{ij}^4| \le \frac{1}{n}$.
        \item The second eigenvalue of the adjacency matrix of a \emph{random $d$-regular graph} on $n$ vertices, as drawn from the configuration model.
    \end{enumerate}
    In all three settings, we obtain constant-probability eigenvalue bounds that are tight up to a constant. Applied to specific random matrix ensembles, we recover classic bounds for Wigner matrices, as well as results of Bordenave--Chafa\"{i}--Garc\'{i}a-Zelada, Bordenave--Lelarge--Massouli\'{e}, and Friedman, up to constants.
\end{abstract}

\input{intro}
\input{complex}
\input{girko}

\input{symm-nb}
\input{wigner}
\input{dreg}

\section*{Acknowledgments}
We would like to thank Kuikui Liu for insightful and inspiring discussions. We are also grateful to Brice Huang, Daniel Lee, and Francisco Pernice for early discussions about our techniques.

\bibliographystyle{alpha}
\bibliography{main}

\end{document}

%% file: intro.tex

\section{Introduction}

Given a matrix $M$, we use $\spec(M)$ to denote its spectrum, and $\rho(M)$ to denote its \emph{spectral radius}, defined by
\[
    \rho(M) \coloneqq \max_{\lambda\in\spec(M)} |\lambda|\mper
\]
A central theme in random matrix theory is to understand the spectrum of a random matrix, and notably its \emph{spectral radius}.
Over the decades, numerous tools have been developed to understand the spectra of random matrices: the trace moment method \cite{FK81}, the method of resolvents \cite{ELKYY13}, matrix concentration inequalities \cite{Tro15}, the polynomial method \cite{CGVTvH24}, chaining \cite{Tal14}, and more.

In this work, we introduce a new technique to control the spectral radius of a random matrix, based on Jensen's formula from complex analysis.
Our technique is inspired by the recent work of Bencs, Liu, and Regts \cite{BLR25}, which uses Jensen's formula to study the zeros of partition functions of spin glasses.

Concretely, as we will see in \Cref{sec:complex-prelims}, the following fact is a straightforward consequence of Jensen's formula.
\begin{restatable}{theorem}{jensendjensen} \label{thm:jensend-jensen}
    For any matrix $M$ and $r \in \R_{ > 0}$, we have:
    \[
        \prod_{\lambda\in\spec(M):\, \lambda > \tau} \parens*{\frac{|\lambda|}{\tau}}^2 \le \E_{\btheta\sim[0,2\pi]} \abs{ \det(I - \frac{e^{\im\btheta}}{\tau}\cdot M ) }^2\mper  \numberthis \label{eq:core-inequality}
    \]
\end{restatable}



This suggests a natural approach for controlling the spectral radius, and the number of outliers of a random matrix $\bM$ drawn from a ``nice'' ensemble $\calD$: explicitly compute and bound the value of $\E_{\bM\sim\calD}\abs{ \det(I - z\bM) }^2$. Concretely, in the setting of a random matrix, we have the following corollaries of \Cref{thm:jensend-jensen}.

\begin{corollary}
    \label{cor:specrad-bound}
    $\displaystyle\E_{\bM \sim \calD} \rho(\bM)^2 \le \tau^2 \cdot \E_{\btheta \sim [0,2\pi]} \E_{\bM \sim \calD} \left| \det\left( \Id - \frac{e^{\im\btheta}}{\tau} \cdot \bM \right) \right|^2\mper$
\end{corollary}

\begin{corollary}
    \label{cor:num-outliers-bound}
    Given a matrix $M$ and positive numbers $\tau$ and $\delta$, let $k_M$ be the number of ``outlier'' eigenvalues $\lambda \in \spec(M)$ such that $|\lambda| > \tau \sqrt{1+\delta}$. Then,
    \[ \E_{\bM \sim \calD} (1+\delta)^{k_{\bM}} \le \E_{\btheta \sim [0,2\pi]} \E_{\bM \sim \calD} \left| \det \left( \Id - \frac{e^{\im\btheta}}{\tau} \cdot \bM \right) \right|^2 \mper \]
\end{corollary}

One may then use Markov's inequality on the above bounds to obtain bounds on the spectral radius (or number of outlier eigenvalues) that hold with constant probability.

As it turns out, in many cases, the combinatorial calculations involved in obtaining a bound on this quantity are simpler than other methods, especially in the sparse regime.

We discuss below the spectral radius bounds we prove in this paper using this method.

\parhead{Random matrices with independent entries.}
Our first result concerns the spectral radius of (asymmetric) matrices with independent entries.
We say that a random matrix $\bM$ is a \emph{non-asymptotic Girko matrix} if it has independent (but not necessarily identically distributed) complex-valued entries: for every $i\ne j\in[n]$, we have $\E \bM_{i,j} = 0$ and $\E \abs*{\bM_{i,j}}^2 = \frac{1}{n}$, and $\bM_{i,i} = 0$ for all $i\in[n]$.

\begin{restatable}{theorem}{girkomain} \label{thm:main-girko}
For an $n\times n$ non-asymptotic Girko matrix $\bM$, we have:
\begin{itemize}
    \item $\E_{\bM}\, \rho(\bM)^2 \le C$ for an absolute constant $C > 1$.
    \item For every $\eps > 0$, $\E_{\bM} \abs{ \{\lambda:\lambda\in\spec(\bM), |\lambda| > 1+\eps\} } \le C'(\eps)$ for an absolute constant $C'(\eps)$ that depends only on $\eps$.
\end{itemize}
\end{restatable}

\begin{remark}
    In the setting of \emph{Girko matrices}, where one first fixes a mean-$0$ and variance-$1$ random variable $\bX$, and considers the ensemble of $n\times n$ random matrices $\bM$ with independent copies of $\bX/\sqrt{n}$ for $n\to\infty$, it was recently proved by Bordenave, Chafa{\"i}, and Garc{\'i}a-Zelada \cite{BCG22} that the spectral radius of $\bM$ is $1+o_n(1)$ with probability $1-o_n(1)$.
    Observe that this setting does \emph{not} encompass random matrix ensembles where the distribution of the entries is allowed to depend on $n$, as is the case in sparse directed random graph models.
    While we work in a more general setting where the entries of $\bM$ may depend on $n$ in an arbitrary way, our results are necessarily weaker on two fronts: obtaining only an $O(1)$ bound on the spectral radius, rather than $1+o(1)$, and guarantees that hold with constant probability rather than with high probability.
    See \Cref{rem:girko-tightness} for examples of random matrix ensembles that witness these limitations. We consider it an interesting question to investigate whether there are settings where these constant-probability bounds can be boosted to high-probability bounds. For example, in the setting where $\bM$ is a draw from the (normalized) Gaussian orthogonal ensemble, which we shortly discuss, a reasonably straightforward argument based on concentration of Lipschitz functions succeeds at doing so.
\end{remark}

\begin{remark}
    We further point out that the bound on the number of outlier eigenvalues in \Cref{thm:main-girko}(b) is \emph{better} than the bound one would obtain if they show that the empirical spectral distribution weakly converges to some limiting law---indeed, this would only imply that $o(n)$ of the eigenvalues are outliers, not $O(1)$.
\end{remark}

\parhead{Hermitian random matrices with independent entries.}
Our story in the setting where $\bM$ is a Hermitian random matrix needs more setup.
In the sequel, in addition to first and second moment assumptions, we also make a very mild assumption on the fourth moment; nothing interesting is true without such an assumption.
We say that an $n\times n$ random matrix $\bM$ is a \emph{non-asymptotic Wigner matrix} if $\bM$ is Hermitian, its entries above the diagonal are independent, and satisfy $\E \bM_{i,j} = 0$, $\E \abs{ \bM_{i,j} }^2 = \frac{1}{n}$, and $\E \abs{ \bM_{i,j} }^4 \le \frac{1}{n}$.
We further assume that the diagonal of $\bM$ is zero to avoid uninteresting complications.

We control the spectral radius of $\bM$ via its \emph{nonbacktracking matrix} $B_{\bM}$, whose rows and columns are indexed by the set of directed edges of the $n$-vertex complete graph, defined as follows:
\[
    B_{\bM}[ij,k\ell] =
    \begin{cases}
        \bM_{k,\ell} &\text{if }j=k\text{ and }i\ne\ell \\
        0 &\text{otherwise}
    \end{cases}
\]
We prove the following for symmetric independent ensembles.
\begin{theorem} \label{thm:main-Wigner}
    Let $\bM$ be an $n\times n$ non-asymptotic Wigner matrix.
    Then,
    \begin{itemize}
        \item $\E_{\bM} \rho\parens*{B_{\bM}}^2 \le C$ for an absolute constant $C > 1$.
        \item For any $\eps > 0$, there is a constant $C'(\eps) > 1$ such that with probability at least $1-\eps$,
        \[
            \rho(\bM) \le C'(\eps)\cdot\parens*{1+\max_{i\in[n]} \norm*{\bM_i}_2}\mper
        \]
    \end{itemize}
\end{theorem}

\begin{remark}
    One can check that when $\bM$ is a ``nice'' random matrix ensemble, such as when its entries are independent standard gaussians, $\max_{i\in[n]} \norm*{\bM_i}_2$ concentrates extremely tightly around $1$, thus recovering the known bound of $2$ on the operator norm up to a constant factor.

    More generally, our bound is tight up to constant factors since for any Hermitian matrix $M$, $\rho(M)\ge \max_{i\in[n]}\norm*{M_i}_2$.
\end{remark}

\begin{remark}
    The bound on $\E_{\bM}\rho\parens*{B_{\bM}}^2$ recovers the celebrated result of Bordenave, Lelarge, and Massouli{\'e} \cite{BLM15} on the eigenvalues of nonbacktracking matrix of \erdos--\renyi graphs and stochastic block models of constant average degree up to absolute constant factors (cf.\ \cite{FM17}).
\end{remark}


\parhead{Random regular graphs.}
Let $\bG$ be a random $d$-regular graph on $n$ vertices.
The eigenvalues of the adjacency matrix $A_{\bG}$ of $\bG$ have been a subject of intense study in random matrix theory and the study of expander graphs.
It was conjectured by Alon \cite{Alo86} that the Alon--Boppana bound \cite{Nil91} is tight for random $d$-regular graphs---besides the trivial eigenvalue of $d$ corresponding to the all-ones vector, all eigenvalues of $A_{\bG}$ are bounded in magnitude by $2\sqrt{d-1}+o_n(1)$.
Friedman \cite{Fri08} proved this conjecture two decades later in a highly technical tour de force using the trace moment method.
Later, Bordenave \cite{Bor19} gave a significantly simpler, though still far from easy, proof also based on the trace moment method.
Notably, a year ago, Chen, Garza Vargas, Tropp and van Handel \cite{CGVTvH24} gave a significantly simpler proof via a new method they pioneered based on connecting the spectra of random matrices to a series expansion in $1/n$ of trace moments of smooth functions of the adjacency matrix $A_{\bG}$.
A parallel line of work \cite{BHY19,HY24} culminating in a breakthrough of Huang, McKenzie, and Yau \cite{HMY24} studied the \emph{resolvent} to prove that all eigenvalues of $A_{\bG}$ are at most $2\sqrt{d-1}$ (without the $o_n(1)$) with probability $\approx0.69$, along with other detailed information of the random matrix ensemble such as the distribution of the fluctuation of the largest eigenvalues, rigidity of the bulk eigenvalues, and delocalization of eigenvectors.

In this work, we use \Cref{thm:jensend-jensen} to prove a constant probability version of Friedman's theorem with an eigenvalue bound that is tight up to a universal constant.
While the result is quantitatively weaker than the results achieved by all of the aforementioned works, we believe our proof is simpler.
In particular, rather loose arguments to control the expression on the right-hand side of \Cref{eq:core-inequality} when $M$ is the nonbacktracking matrix of a random regular graph lose only a constant factor in the spectral radius.

\begin{theorem} \label{thm:main-dreg}
    Let $\bG$ be a random $d$-regular graph drawn from the configuration model.
    For every $\eps > 0$, there is a constant $C(\eps)$ such that with probability $1-\eps$, $\max\{\lambda_2(A_{\bG}), -\lambda_n(A_{\bG})\}$ is at most $C(\eps)\sqrt{d-1}$.
\end{theorem}

\parhead{Organization.}
In \Cref{sec:complex-prelims}, we state and derive the basic complex analytic tools we will need.
In \Cref{sec:girko}, we will prove \Cref{thm:main-girko}, our main result on the spectral radius of non-asymptotic Girko matrices.
In \Cref{sec:nb-prelims}, we describe the connection between the spectral radii of a matrix and its nonbacktracking matrix, and derive some properties of the determinant of a nonbacktracking matrix.
Finally, in \Cref{sec:Wigner}, we prove \Cref{thm:main-Wigner}, our main result on non-asymptotic Wigner matrices, and in \Cref{sec:mixed-moments-dreg}, we prove \Cref{thm:main-dreg}, our spectral radius bound on the random $d$-regular ensemble.


%% file: complex.tex

\section{Complex analytic facts}
\label{sec:complex-prelims}

The only tool we will need is Jensen's formula from complex analysis.

\begin{theorem}[{Jensen's formula~\cite[Eq. (44), Chapter 5.3.1]{Ahl79}}]
    \label{th:jensen-formula}
    Let $f : \C \to \C$ be holomorphic, and let $a_1,\dots,a_k$ be the zeros (with multiplicity) of $f$ inside the scaled disk $r\bbD$, for some $r > 0$. Then,
    \[
        \E_{\btheta\sim[0,2\pi]} \log \abs*{ f(re^{\im\btheta}) } = \log|f(0)| + \sum_{t=1}^k \log\parens*{\frac{r}{|a_t|}} \mper
    \]
\end{theorem}

In particular, using Jensen's inequality on the above easily yields the following.

\begin{corollary}
    \label{cor:jensen-jensen-formula}
    Let $f : \C \to \C$ be holomorphic, and let $a_1,\dots,a_k$ be the zeros (with multiplicity) of $f$ inside the scaled disk $r\bbD$. Then,
    \[
        |f(0)|^2 \cdot \prod_{t=1}^{k} \left( \frac{r}{|a_t|} \right)^2 \le \E_{\btheta\sim[0,2\pi]} \abs*{ f(r e^{\im\btheta}) }^2 \mper
    \]
\end{corollary}
\begin{proof}
    By Jensen's inequality and Jensen's formula, we have:
    \[ \log \E_{\btheta\sim[0,2\pi]} \abs*{ f(re^{\im\btheta}) }^2 \ge  \E_{\btheta\sim[0,2\pi]} \log \abs*{ f(re^{\im\btheta}) }^2 = 2 \cdot \left( \log|f(0)| + \sum_{t=1}^k \log\parens*{\frac{r}{|a_t|}} \right) \mper \]
    We then obtain the desired statement by exponentiating the above inequality.
\end{proof}

\Cref{thm:jensend-jensen}, which we restate below for convenience, is a special case of \Cref{cor:jensen-jensen-formula} obtained by specializing $f(z) = \det(I-z M)$, and choosing $r = \frac{1}{\tau}$, since the roots of $f$ are $\left\{\frac{1}{\lambda}\right\}_{\lambda\in\spec(M) \mcom \lambda \ne 0}$.

\jensendjensen*

\begin{remark}
    The same inequality holds if we choose $f(z)$ as $\det(I-z M)\cdot g(z)$ for any holomorphic function $g$ that is zero-free on $r\bbD$; in our case $g$ is chosen as the constant function, but in some cases one may obtain mileage by a clever choice of $g$ adapted to the ensemble at hand.
    Indeed, we believe that when studying symmetric matrix ensembles, passing to the nonbacktracking matrix implicitly involves choosing $g$ as some suitable non-constant function.
\end{remark}

%% file: girko.tex

\section{Spectral radius bounds for non-asymptotic Girko matrices}
\label{sec:girko}

In this section, we prove \Cref{thm:main-girko}, restated below.
\girkomain*

\begin{proof}[Proof of \Cref{thm:main-girko}]
    By \Cref{cor:specrad-bound,cor:num-outliers-bound}, to prove both parts, it suffices to prove for any $\tau = 1+\eps$:
    \[
        \E_{\bM} \E_{\btheta\sim[0,2\pi]} \abs{ \det(I-\frac{e^{\im\btheta}}{\tau} \cdot \bM) }^2 \le O_{\eps}(1)\mper
    \]
    Using the fact that the diagonal of $\bM$ is equal to $0$, we have for any $z$ such that $|z| = \frac{1}{\tau}$:
        \begin{align*}
            \E_{\bM} \abs{ \det(I-z\cdot \bM) }^2 &= \E_{\bM} \det\left( \Id - \bM z \right) \cdot \det\left( \Id - \bM^* z^* \right) \\
            &= \sum_{\pi,\sigma \in S_n}  \E_{\bM} \sgn(\sigma) \sgn(\pi) \cdot \left( \prod_{\substack{i \in [n] \\ \sigma(i) \ne i}} - \bM_{i\sigma(i)} z \right) \cdot \left( \prod_{\substack{i \in [n] \\ \pi(i) \ne i}} - \bM_{i\pi(i)}^* z^* \right) \mper
    \intertext{Observe that a term corresponding to $(\pi,\sigma)$ in the above summation vanishes if there is some $i \in [n]$ such that $\sigma(i) \ne \pi(i)$.
    Therefore, denoting $\NF(\sigma) = \{i \in [n] : \sigma(i) \ne i\}$, we can continue the above chain of equalities as:}
       &= \sum_{\pi \in S_n} \tau^{-2|\NF(\pi)|} \cdot \prod_{i \in \NF(\pi)} M_{i\pi(i)}^2 \\
        &= \sum_{\pi \in S_n} \left(\tau^2 n\right)^{-\NF(\pi)} \\
        &= \sum_{k \ge 0} \left(\tau^2 n\right)^{-k} \cdot \left|\left\{ \pi \in S_n : \NF(\pi) = k \right\}\right| \\
        &= \sum_{k \ge 0} \left(\tau^2 n\right)^{-k} \cdot \binom{n}{k} \cdot D_k \mcom
    \intertext{where $D_k$ is the number of derangements on $[k]$, that is, the number of permutations $\sigma \in S_k$ such that $\sigma(i) \ne i$ for all $i$. Clearly, $D_k \le k!$ for all $k$. Thus, the above can be bounded by}
            &\le \sum_{k \ge 0} \left( \tau^2 n \right)^{-k} \cdot \frac{n!}{(n-k)!} \\
            &\le \sum_{k \ge 0} \tau^{-2k} \\
            &= \frac{\tau^2}{\tau^2 - 1} \mper \qedhere
    \end{align*}
    The above bound combined with \Cref{cor:specrad-bound} implies the desired bound on $\E_{\bM}\rho(\bM)^2$ by choosing $\tau$ as any constant larger than $1$.
    The bound on $\E_{\bM} \abs{ \{ \lambda: \lambda\in\spec(\bM),\, \abs{\lambda} > 1 + \eps \} }$ follows from the above bound combined with \Cref{cor:num-outliers-bound} by choosing $\tau$ as $\sqrt{1+\eps}$, and $\delta = \eps$.
\end{proof}

\begin{remark}[On tightness of bounds]
    \label{rem:girko-tightness}
    Most past study on Girko matrices has been in the setting where the distributions of the entries of the matrix do not depend on $n$. 
    More concretely, they assume some ensemble $(\ba_{ij})_{i,j \ge 1}$ of iid centered unit variance random variables, and set $\bA = \left( \frac{1}{\sqrt{n}} \ba_{ij}\right)_{1 \le i,j \le n}$.
    In particular, it was shown that the empirical spectral distribution of $\bA$ weakly converges to the circular law $\Unif(\bbD)$ \cite{Gir85,Gir18,TVK10}. Later, in \cite{BCG22}, it was proved that in fact, for any $\eps > 0$, $\Pr\left[ |\rho(\bA)-1| > \eps \right] = o(1)$. The latter is significantly stronger than the weak ``Markov-esque'' tails we obtain in \Cref{thm:main-girko}.

    However, such statements are \emph{false} in the generality we assume, where the $\ba_{ij}$ are further allowed to depend on $n$. For example, consider the matrix $\bM$ whose entries $\bM_{ij}$ are independently distributed as
    \[
    \bM_{ij} =
    \begin{cases}
        0 \mcom & \text{w.p. } 1 - 2^{-n} \mcom \\
        \frac{1}{\sqrt{n}} \cdot 2^{n/2} \mcom & \text{w.p. } 2^{-n-1} \mcom \\
        -\frac{1}{\sqrt{n}} \cdot 2^{n/2} \mcom & \text{w.p. } 2^{-n-1} \mper
    \end{cases}
    \]
    Clearly, the entries of $\bM$ are centered and have variance $\frac{1}{n}$. However, with $1-o(1)$ probability, $\bM = 0$, so its empirical spectral distribution in fact converges to $\delta_0$, the Dirac measure on $0$.\\
    At the level of generality we work in, it is also not true that for every $\eps > 0$, $\Pr\left[ \rho(\bA) > 1+\eps \right] = o(1)$.
    For example, consider the appropriately scaled centered adjacency matrix of a draw $G$ from the (directed) \erdos--\renyi~graph ensemble, that is,
    \[ \frac{1}{\sqrt{2}} \cdot \bM_{ij} =
    \begin{cases}
        -\frac{1}{2n} \mcom & \text{w.p. } 1 - \frac{1}{2n} \mcom \\
        1 - \frac{1}{2n} \mcom & \text{w.p. } \frac{1}{2n} \mper
    \end{cases}
    \]
    The entries of this matrix have variance approximately $\frac{1}{n}$.
    It is easy to check that with $\Omega(1)$ probability, $\rho(\bM) \gtrapprox \sqrt{2}$.
    Indeed, with $\Omega(1)$ probability, there is a $3$-cycle $abc$ in $G$, such that $a,b,c$ further have no edges to any other vertices in $G$. In this case, the spectral radius of $\bM$ is bounded from below by the spectral radius of the principal submatrix on $\{a,b,c\}$, which is given by
    \[ \sqrt{2} \begin{pmatrix} 0 & 1 - \frac{1}{2n} & -\frac{1}{2n} \\ -\frac{1}{2n} & 0 & 1 - \frac{1}{2n} \\ 1 - \frac{1}{2n} & -\frac{1}{2n} & 0 \end{pmatrix} \mper \]
    However, for sufficiently large $n$, the spectral radius of this matrix is $\approx \sqrt{2} > 1$, so the spectral radius of $\bM$ is also bounded away from $1$ with constant probability.
\end{remark}

%% file: symm-nb.tex
\section{Hermitian matrices and the nonbacktracking matrix} \label{sec:nb-prelims}

In this section, we review the well-understood connection between the spectral radius of a Hermitian matrix and its nonbacktracking matrix, and also establish some properties of nonbacktracking matrices that we use in our proof.

In particular, we use the following fact, which may be proved using the remarkable Ihara--Bass formula \cite{Iha96,Bas92}.

\begin{lemma}[Consequence of {\cite[Theorem 2.2]{BGBK20}}]   \label{lem:ib-consequence}
    Let $M$ be an $n\times n$ Hermitian matrix, and let $B_M$ be its nonbacktracking matrix.
    Then:
    \[
        \rho(M) \le 2 \rho(B_M) + 9\cdot \max_i \norm*{M_i}_2\mper
    \]
\end{lemma}

\begin{remark}
    The translation between the spectral radius and that of its nonbacktracking matrix applies to any matrix and hence is crude, and lossy by a constant factor.
    For many ``nice'' random matrix ensembles, there is essentially no loss incurred in passing to the nonbacktracking matrix.
\end{remark}

\begin{remark}
    While the above statement only translates bounds on the spectral radius of the nonbacktracking matrix to bounds on the spectral radius of $M$, the Ihara--Bass formula may similarly be used to translate bounds on the number of outlier eigenvalues (of the form in \Cref{thm:main-girko}(b)). We omit the details; such bounds immediately follow from all the computations we perform here, when used in conjunction with \Cref{cor:num-outliers-bound}.
\end{remark}

For a random matrix $\bM$ in the wild, it is often easy to control $\max_i\norm*{\bM_i}_2$, and the challenging part is in getting a handle on $\rho(B_{\bM})$.
Today, we will use \Cref{cor:specrad-bound} to control $\rho(B_{\bM})$ by studying $\det(I - B_{\bM}z)$ where $\bM$ is either a non-asymptotic Wigner matrix, or the adjacency matrix of a random regular graph.

We use $\vec{E}$ to denote the set of directed edges on the $n$-vertex complete graph, and $S_{\vec{E}}$ to denote the set of all permutations on $\vec{E}$.
We will use $\NBP_{\vec{E}}\subseteq S_{\vec{E}}$ to refer to the set of all \emph{nonbacktracking permutations}, i.e., permutations $\pi\in S_{\vec{E}}$ such that for any $ij\in\vec{E}$, $\pi(ij)$ is either equal to $ij$, or is equal to $jk$ for some $k\ne i$.
For a permutation $\pi$, we use $\NTCyc(\pi)$ to denote the number of nontrivial cycles in the permutation $\pi$, i.e., cycles of length at least $2$, and we use $\Cyc(\pi)$ to denote the number of cycles in $\pi$ (including trivial ones).

We have:
\begin{align*}
	\det\left( \Id - B_{\bM}z \right) &= \sum_{\pi \in S_{\vec{E}}} \sgn(\pi) \prod_{e \in \NF(\pi)} (-B_{e \pi(e)} z) \\
	&= \sum_{\pi \in \NBP_{\vec{E}}} \sgn(\pi) \cdot \prod_{e \in \NF(\pi)} (-B_{e \pi(e)} z) \\
    &= \sum_{\pi\in\NBP_{\vec{E}}} (-1)^{\NTCyc(\pi)} \cdot z^{|\NF(\pi)|} \cdot \prod_{e\in\NF(\pi)} \bM_{\pi(e)}\mper
\end{align*}
Using $\calH$ to denote the set of all directed subgraphs on $[n]$ (that is, subsets of $\vec{E}$), and for $H\in\calH$, using $\NBP_H$ to denote the set of all $\pi\in\NBP_{\vec{E}}$ such that $\NF(\pi) = H$, we may write the above as:
\[
    \det( \Id - B_{\bM}z) = \sum_{H\in\calH} z^{e(H)} \cdot \prod_{e\in H} \bM_e \cdot \sum_{\pi\in\NBP(H)} (-1)^{\NTCyc(\pi)}  \numberthis \label{eq:det-subgraph-decomp}
\]
We first substantially simplify the sum over elements in $\NBP(H)$.
Towards doing so, let us introduce some terminology.
\begin{definition}  \label{def:nbp-simple-notation}
    For $H\in\calH$, we say that $e\in H$ is a \emph{doubleton} if the reverse of $e$ also occurs in $H$, and we say $e$ is a \emph{singleton} otherwise.
    For a vertex $v\in[n]$, we use $\In_H(v)$ to denote the set of all incoming edges to $v$ in $H$, and $\Out_H(v)$ to denote the set of all outgoing edges.
    Let $\wt{\calH}\subseteq\calH$ be the set of all $H\in\calH$ such that every vertex has at most one incoming singleton edge and at most one outgoing singleton edge.
\end{definition}
A simple yet important observation is the following.
\begin{observation}
    Let $\pi\in\NBP(H)$.
    By the fact that $\pi$ is a nonbacktracking permutation, for every vertex $v$ in $[n]$, $d_v^{H} \coloneqq \abs{\In_H(v)} = \abs{\Out_H(v)}$, and $\pi(\In_H(v)) = \Out_H(v)$.
\end{observation}
We prove the following simplification for the determinant.
\begin{lemma}   \label{lem:simple-form-det}
   For any matrix $M$, we have
    \[
        \det(\Id - B_{M}z) = \sum_{H\in\wt{\calH}} z^{e(H)} \cdot \prod_{e\in H} M_e \cdot \sum_{\pi\in T(H)} (-1)^{\NTCyc(\pi)}
    \]
    for some $T(H)\subseteq \NBP(H)$ with $|T(H)| \le \prod_{v\in[n]} d_v^{H}$.
\end{lemma}
\begin{proof}
    For $\pi\in\NBP(H)$ and $v\in[n]$, we use $\pi_v:\In_H(v)\to\Out_H(v)$ to refer to the restriction of $\pi$ to the edges incident to $v$.
    We first observe that $\pi$ is uniquely determined by a collection of local bijections $(\pi_v)_{v\in[n]}$.
    In particular, $\pi$ is forced to be the permutation where $\pi(uv) = \pi_v(uv)$.
    We also observe that if each local permutation obeys the nonbacktracking constraint, $\pi$ is in $\NBP(H)$.
    We now set up some notation.
    \begin{itemize}
        \item Let $\Omega_v$ be the collection of all nonbacktracking local bijections from $\In_H(v)$ and $\Out_H(v)$.
        \item Let $\Omega \coloneqq \prod_{v\in[n]}\Omega_v$, and let $\Omega_{-v} = \prod_{u\in[n]\setminus v} \Omega_u$.
        \item Let $\pi_{-v}$ denote $(\pi_u)_{u\in[n]\setminus v}$.
        \item Let $\NTCyc_v(\pi)$ refer to the number of nontrivial cycles in $\pi$ that pass through $v$, and let $\NTCyc_{-v}(\pi)$ refer to the number of cycles in $\pi$ that do not pass through $v$.
    \end{itemize}
    For any $v\in[n]$, we can write:
    \begin{align*}
        \sum_{\pi\in\NBP(H)} (-1)^{\NTCyc(\pi)} &= \sum_{\pi_1,\dots,\pi_n\in\Omega} (-1)^{\NTCyc(\pi)} \\
        &= \sum_{\pi_{-v}\in\Omega_{-v}} \sum_{\pi_v \in \Omega_v} (-1)^{\NTCyc_{-v}(\pi) + \NTCyc_v(\pi)} \\
        &= \sum_{\pi_{-v}\in\Omega_{-v}} (-1)^{\NTCyc_{-v}(\pi_{-v})} \sum_{\pi_v\in\Omega_v} (-1)^{\NTCyc_v(\pi)},
    \end{align*}
    where in the last equality uses the fact that $\NTCyc_{-v}(\pi)$ does not depend on $\pi_v$.
    Observe that one can extract from $\pi_{-v}$ a bijection $\sigma:\Out_H(v)\to\In_H(v)$ where $\sigma(vw)$ is equal to $w'v$ obtained by following the path starting at $vw$ described by $\pi_{-v}$ until it hits $v$ next.
    We may treat $\sigma$ and $\pi_v$ as permutations on $[d_v^{\pi}]$ by arbitrarily labeling the elements of $\Out_H(v)$ and $\In_H(v)$ with elements of $[d_v^{\pi}]$.
    Next, observe that $\NTCyc_v(\pi) = \Cyc(\sigma\circ\pi_v)$, and so,
    \[
        (-1)^{\NTCyc_v(\pi)} = (-1)^{\Cyc(\sigma\circ\pi_v)} = (-1)^{d_v^{\pi}} \cdot \sgn(\sigma\circ\pi_v) = (-1)^{d_v^{\pi}} \cdot \sgn(\sigma) \cdot \sgn(\pi_v)\mper
    \]
    Since $\sigma$ depends only on $\pi_{-v}$, we get:
    \[
        \sum_{\pi\in\NBP(H)} (-1)^{\NTCyc(\pi)} = \sum_{\pi_{-v}\in\Omega_{-v}} (-1)^{\NTCyc_{-v}(\pi_{-v})}\cdot \sgn(\sigma) \cdot (-1)^{d_v^{\pi}} \cdot \sum_{\pi_v\in\Omega_v} \sgn(\pi_v)\mper \numberthis \label{eq:disentangled-Omega_v}
    \]
    The innermost sum is taken over all permutations subject to the constraint that $\pi_v$ cannot map an in-edge to the out-edge that reverses it.
    Thus, for the $d_v^{\pi}\times d_v^{\pi}$ matrix $R$ where $R_{ij}$ is $0$ if the $i$-th in-edge is the reversal of the $j$-th outedge, and place $1$ otherwise, we have that
    \[
        \sum_{\pi_v\in\Omega_v} \sgn(\pi_v) = \det(R)\mper
    \]
    Observe that $R$ is a $\{0,1\}$-matrix where every row and column has at most a single $0$, and so we can understand its determinant based on the following cases:
    \begin{itemize}
        \item {\it Number of rows that are all-ones is at least $2$.} In this case, $\det(R) = 0$ since the matrix is singular by virtue of having repeated rows.
        \item {\it Number of rows that are all-ones is exactly $1$.} In this case, $\det(R) \in \{ \pm \det(11^{\top} - \Id + E_{11})\}$, which is equal to $\{\pm 1\}$ via explicit calculation of eigenvalues.
        \item {\it Every row/column has a $0$.}  In this case, $\det(R) \in \{\pm\det(11^{\top}-\Id)\}$, which is equal to $\{\pm(d_v^{\pi}-1)\}$ via explicit calculation of eigenvalues.
    \end{itemize}
    The above casework has a few upshots.
    First, if $H$ is not in $\wt{\calH}$, then we choose $v$ to be a vertex with either more than one incoming singleton or outgoing singleton edge.
    For such a vertex, we must have $ \sum_{\pi_v\in\Omega_v} \sgn(\pi_v) = 0$, and so the entire term corresponding to $H$ in \Cref{eq:det-subgraph-decomp} vanishes.
    Thus, we may assume $H$ is in $\wt{\calH}$.
    Since $\abs{\sum_{\pi_v\in\Omega_v} \sgn(\pi_v)} \le d_v^{\pi}$, and each term in the summand is a sign, we may fix a subset $T_v\subseteq\Omega_v$ of at most $d_v^{\pi}$ elements such that $\sum_{\pi_v\in T_v} \sgn(\pi_v) = \sum_{\pi_v \in \Omega_v} \sgn(\pi_v)$.
    As a consequence, we have:
    \[
        \sum_{\pi\in\NBP(H)} (-1)^{\NTCyc(\pi)} = \sum_{\pi_{-v}\in\Omega_{-v}} \sum_{\pi_v\in T_v} (-1)^{\NTCyc(\pi)}\mper
    \]
    We may iteratively replace each $\Omega_v$ with $T_v$ and obtain:
    \[
        \sum_{\pi\in\NBP(H)} (-1)^{\NTCyc(\pi)} = \sum_{\pi_1,\dots,\pi_n\in T_1\times\dots\times T_n} (-1)^{\NTCyc(\pi)}\mper
    \]
    The desired statement follows from choosing $T(H) = T_1\times \dots \times T_n$, and our bound on all $|T_v|$.
\end{proof}


%% file: wigner.tex

\section{Spectral radius bound for the nonbacktracking matrix}  \label{sec:Wigner}

In this section, we prove that for a random matrix $\bM$ satisfying fairly reasonable bounds on mixed moments in its entries, we can obtain a spectral radius bound on $B_{\bM}$ via \Cref{cor:specrad-bound}.
Our bounds on the spectral radius of non-asymptotic Wigner matrices and random regular graphs will then immediately follow from verifying that they satisfy the requisite condition on their mixed moments.

For the sequel we will need the following notation.
\begin{definition}[Notation for undirecting a directed graph]
    Given a directed graph $H$, we use $\UD(H)$ to denote the undirected (multi-)graph obtained by taking each directed edge $uv$ in $H$, and replacing it with an undirected edge between $u$ and $v$.
    We use $\SUD(H)$ to denote the \emph{simple} undirected graph obtained by taking $\UD(H)$ and replacing each multi-edge with a single edge.
\end{definition}

\begin{assumption}  \label{ass:mixed-moment-bound}
    For every directed subgraph $H\in\wt{\calH}$ (for $\wt{\calH}$ defined in \Cref{def:nbp-simple-notation}), and every assignment of ``multiplicities'' $m:E(\SUD(H))\to\{1,2,3,4\}$ to the edges of $\SUD(H)$, we have:
    \[
        \E_{\bM} \prod_{ij\in E(\SUD(H))} \bM_{ij}^{m(ij)} \le \left(\frac{C}{n}\right)^{|E(H)|}
    \]
    for an absolute constant $C > 1$ independent of $n$.
\end{assumption}

\begin{theorem} \label{thm:moment-assumption-to-spec-norm}
    There exists an absolute constant $\alpha > 0$ such that for any $n\times n$ random matrix $\bM$ satisfying \Cref{ass:mixed-moment-bound}, we have:
    \[
        \E_{\bM} \rho(\bM)^2 \le \alpha\mper
    \]
\end{theorem}

\begin{remark}
    \Cref{thm:main-Wigner} follows immediately from \Cref{thm:moment-assumption-to-spec-norm} since a non-asymptotic Wigner matrix can be readily verified to satisfy \Cref{ass:mixed-moment-bound} using independence.
    Verifying \Cref{ass:mixed-moment-bound} for the random $d$-regular ensemble is a more challenging task, and we do so in \Cref{sec:mixed-moments-dreg}.
\end{remark}

We now prove \Cref{thm:moment-assumption-to-spec-norm} below.
\begin{proof}[Proof of \Cref{thm:moment-assumption-to-spec-norm}]
    By \Cref{cor:specrad-bound}, it suffices to prove an upper bound on
    \[
        \E_{\bM} \abs*{ \det(\Id - z \cdot B_{\bM}) }^2
    \]
    for $|z| = 1/\tau$ for some sufficiently large constant $\tau$.
    We apply \Cref{lem:simple-form-det} and obtain:
    \begin{align}
        \E_{\bM} \abs*{ \det(\Id - z \cdot B_{\bM}) }^2 &=
        \E_{\bM} \det(\Id - z \cdot B_{\bM}) \cdot \det(\Id - z^* \cdot B_{\bM^*}) \nonumber \\
        &\le \sum_{H,H'\in\wt{\calH}} \tau^{-e(H)-e(H')} \cdot \abs{\E_{\bM} \prod_{e\in \vec{E}(H)} \bM_e \cdot \prod_{e\in\vec{E}(H')} \bM_e} \cdot \abs{\sum_{\substack{\pi\in T(H) \\ \pi'\in T(H')}} (-1)^{\NTCyc(\pi) + \NTCyc(\pi')}} \nonumber \\
        &\le \sum_{H,H'\in\wt{\calH}} \parens*{\frac{C}{\tau n}}^{e(\SUD(H\cup H'))} \cdot |T(H)| \cdot |T(H')| \nonumber \\
        \intertext{where the third inequality used \Cref{ass:mixed-moment-bound}. \nonumber
        We may continue the above chain of inequalities below.}
        &\le \sum_{k\ge 0} \sum_{\substack{H,H'\in\wt{\calH} \\ e(\SUD(H,H')) = k}} \parens*{ \frac{C}{\tau n} }^k \cdot |T(H)| \cdot |T(H')| \tag{$\star$} \label{eq:subgraph-count-reduction}
    \end{align}
    First, observe that for any $H,H'\in\wt{\calH}$, all non-isolated vertices of the graph $R\coloneqq\SUD(H,H')$ have degree at least $2$, and thus $e(R)\ge v(R)$.
    We relax the above sum to enumerate over \emph{all} $k$-edge graphs $R$ where all non-isolated vertices of $R$ have degree at least $2$, and over \emph{all} pairs of directed graphs $H,H'$ such that $\SUD(H\cup H') = R$.
    For every edge $\{u,v\}\in E(R)$, given a subset of the statements $\{uv\in H, vu\in H, uv\in H', vu\in H'\}$, it is possible to recover $H$ and $H'$.
    Thus, given $R$, there are at most $16^k$ choices for $H$ and $H'$.
    Next, we use the bound on $|T(H)|$ and $|T(H')|$ from \Cref{lem:simple-form-det} to observe that:
    $$
        |T(H)| \cdot |T(H')| \le \prod_{v\in V(R)} \deg_R(v)^2 \le \parens*{\frac{\sum_v \deg_R(v)}{r}}^{2r} = \parens*{ \frac{2k}{r} }^{2r} \le O(1)^k \mcom
    $$
    since $(2k/r)^{r/2k}$ is uniformly bounded.
    Consequently, using $\calR(r,k)$ to denote the set of all $k$-edge subgraphs of the complete graph with $r$ nonisolated vertices, we may write
    \begin{align*}
        \text{\eqref{eq:subgraph-count-reduction}} &\le \sum_{r\ge 0} \sum_{k\ge r}\, \sum_{R\in\calR(r,k)} \parens*{ \frac{O(1)}{\tau n} }^k \\
        &\le \sum_{r\ge 0} \sum_{k\ge r} |\calR(r,k)| \parens*{\frac{O(1)}{\tau n}}^k \\
        &\le \sum_{k\ge 0} \sum_{0\le r\le \min\{k,n\}} {n\choose r} {r^2/2 \choose k} \parens*{ \frac{O(1)}{\tau n} }^k \\
        &\le \sum_{k\ge 0} \tau^{-k} \cdot O(1)^k \cdot \sum_{0\le r \le \min\{k,n\}} \parens*{ \frac{en}{r} }^r \cdot \parens*{ \frac{er^2}{2k} }^k \cdot \frac{1}{n^k} \\
        &\le \sum_{k\ge 0} \tau^{-k} \cdot O(1)^k \cdot \sum_{0\le r\le \min\{k,n\}} \parens*{\frac{r}{n}}^{k-r} \\
        &\le \sum_{k\ge 0} \tau^{-k} \cdot O(1)^k\mcom
    \end{align*}
    where the third inequality used that one can construct an $r$-vertex $k$-edge graph by picking $r$ vertices and then choosing any subset of $k$ edges, and the second last inequality used $r \le k$.
    Choosing $\tau$ as a sufficiently large constant so as to induce geometric decay in the sum completes the proof.
\end{proof}

%% file: dreg.tex

\section{On mixed moments of entries of random regular graphs}	\label{sec:mixed-moments-dreg}



In this section, we shall prove \Cref{ass:mixed-moment-bound} in the setting where $\bM$ is the normalized and centered adjacency matrix of a random $d$-regular graph. Since we are okay losing constant factors in the spectral radius bound, we may assume that $d$ is some sufficiently large constant.
The specific model of random $d$-regular graphs we study is the configuration model.

\begin{definition}[Configuration model]
	Given $n \ge 1$ and $d \ge 2$, a draw $\bG$ from the corresponding configuration model is defined as follows. Let $\calG$ be a uniformly random perfect matching on $[n] \times [d]$, and let $\bG$ be the multigraph on $[n]$ obtained by adding an edge $ij$ for each edge in $\calG$ between the corresponding ``clouds'' $\{i\} \times [d]$ and $\{j\} \times [d]$.
	Let $\bA$ be the adjacency matrix of $\bG$, and $\bM$ the random matrix defined by
	\[ \bM_{ij} =
	\begin{cases}
		\frac{1}{\sqrt{d}} \cdot \left(\bA_{ij} - \frac{d}{n}\right) \mcom & i \ne j \mcom \\
		0 \mcom & i = j \mper
	\end{cases}
	\]
\end{definition}

\begin{theorem}
	\label{th:ass-moment-ass-dreg-ass}
	\Cref{ass:mixed-moment-bound} holds for $\bM$ defined as above.
\end{theorem}

\begin{proof}[Proof of \Cref{thm:main-dreg}]
	\Cref{th:ass-moment-ass-dreg-ass,thm:moment-assumption-to-spec-norm} yield a constant-probability bound on the spectral radius of the nonbacktracking matrix $B_{\bM}$ associated to $\bM$ of an absolute constant $C$.
	It remains to prove that the maximum row norm of $\bM$ is bounded by $O(1)$ with high probability; this would translate back to a bound of $O(1)$ on the largest eigenvalue of $\bM$ via \Cref{lem:ib-consequence}.

	First, observe that $\norm*{\bM_i} \le \frac{\norm*{(A_{\bG})_i}}{\sqrt{d}} $, and so it suffices to prove $\norm{(A_{\bG})_i} \le O(\sqrt{d})$ with probability at least $1-o(1/n)$.
	We split into cases: when $d \le \sqrt{n}$, the bound of the maximum row norm follows from the fact that $A_{\bG}$ has at most $d$ nonzero entries per row, and that all entries of $A_{\bG}$ are at most $5$ with probability $1-o(1)$.

	We now treat the case when $n \ge d > \sqrt{n}$.
	The random variable $\bx$ describing a row of $A_{\bG}$ can be modeled as the histogram resulting from tossing $d$ balls into $n$ bins.
	We can couple $\bx$ with a random variable $\by$ obtained by tossing $\Poi(2d)$ balls into $n$ bins (or equivalently a vector of $n$ independent draws from $\Poi(2d/n)$) such that $\by \ge \bx$ with probability $1-o(1/n)$.
	One may prove a bound on $\norm*{\by}$ of $O(\sqrt{d})$ that holds with probability $1-o(1/n)$ by proving $\E(\norm*{\by}^2 - \E\norm*{\by}^2)^k \le k! \cdot d^{k/2}$; we omit the details.
\end{proof}

Key to the proof will be establishing concrete wins that can be gained when many of the multiplicities $m(ij)$ are equal to $1$---indeed, in the Wigner setting from the previous section, any monomial with a variable that appears with degree-$1$ (henceforth referred to as a \emph{singleton} edges) is zero in expectation.

Our proof relies on estimates for various centered (and uncentered) expectations in the configuration model, starting with a relatively trivial bound that does \emph{not} see any of the extra wins from singleton edges. In the next few lemmas, one must think of $N = nd$.

\begin{lemma}
	\label{lem:trivial-dreg-bound}
	Let $S$ be an arbitrary subgraph on vertex set $[N]$. Let $m_e \ge 1$ be some constant independent of $N$ for each $e \in S$. Then, for $\calG_{[N]}$ a uniformly random perfect matching on $[N]$, we have
	\[ \left| \E \prod_{e \in S} \left( \bone_{e \in \calG_{[N]}} - \frac{1}{N} \right)^{m_e} \right| \le O\left(\frac{1}{N}\right)^{|S|} \mper \]
\end{lemma}
\begin{proof}
	We may simply expand out the product to write
	\begin{align*}
		\left| \E \prod_{e \in S} \left( \bone_{e \in \calG_{[N]}} - \frac{1}{N} \right)^{m_e} \right| &\le \sum_{R \subseteq S} \left( \frac{1}{N} \right)^{|S \setminus R|} \cdot \E \bone_{R \subseteq \calG_{[N]}} \\
			&\le \sum_{R \subseteq S} \left( \frac{1}{N} \right)^{|S \setminus R|} \cdot \frac{(N-2|R|-1)!!}{(N-1)!!} = O\left( \frac{1}{N} \right)^{|S|} \mper
	\end{align*}
\end{proof}

Now, on the other end of the spectrum, let us demonstrate the wins that appear when every edge is a singleton in the configuration model. To do so, we shall use the Laplace method.

\begin{lemma}[{Laplace's Method, \cite[Chapter 2]{But07}}]
	\label{lem:laplace-method}
	Let $f,g : \R \to \R$ be such that $g$ has a unique global maximizer $x^\star$ with $g$ analytic in a neighborhood of $x^\star$, with $f(x^\star) \ne 0$ and $g''(x^\star) \ne 0$. Then,
	\[ \int_{-\infty}^{\infty} f(x) e^{Ng(x)} \dif x = f(x^\star) e^{Ng(x^\star)} \sqrt{\frac{2\pi}{Ng''(x^\star)}} \left( 1 + O\left(\frac{1}{N}\right) \right)\mper \]
	Here, the $O\left(\frac{1}{N}\right)$ hides factors depending on the third and fourth derivatives of $f$ and $g$ at $x^\star$.
\end{lemma}

\begin{lemma}
	\label{lem:dreg-laplace}
	Let $S$ be a matching with $k$ edges on vertex set $[N]$, and $\beta \ge 1$. For simplicity, assume that $k < \frac{N}{4}$. If $\calG$ is a uniformly random perfect matching on $[N]$, we have
	\[ \left| \E \prod_{e \in S} \left( \bone_{e \in \calG} - \frac{1}{\beta N} \right) \right| \le O\left( \frac{1}{N} \right)^{k} \cdot \left( \frac{\beta-1}{\beta} + \sqrt{\frac{2k}{\beta N}} \right)^{k} \mper \]
\end{lemma}
\begin{proof}
	We start by expanding
	\begin{align*}
		\E \prod_{e \in S} \left( \bone_{e \in \calG} - \frac{1}{\beta N} \right) &= \sum_{T \subseteq S} \E \bone_{T \subseteq \calG} \cdot \left( -\frac{1}{\beta N} \right)^{k - |T|}
	\end{align*}
	By the definition of $\calG$, we have
	\[ \Pr\left[ T \subseteq \calG \right] = \frac{(N-2|T|-1)!!}{(N-1)!!} \mper \]
	Thus, the above expectation may be written as
	\begin{align*}
		\E \prod_{e \in S} \left( \bone_{e \in \calG} - \frac{1}{N} \right) &= \sum_{T \subseteq S} \Pr\left[ T \subseteq \calG \right] \cdot \left( - \frac{1}{\beta N} \right)^{k-|T|} \\
			&= \sum_{T \subseteq S} \frac{(N-2|T|-1)!!}{(N-1)!!} \cdot \left( - \frac{1}{\beta N} \right)^{k - |T|} \\
			&= \sum_{0 \le r \le k} \binom{k}{r} \cdot \frac{(N-2|T|-1)!!}{(N-1)!!} \cdot \left( - \frac{1}{\beta N} \right)^{k - |T|} \\
			&= \frac{1}{(N-1)!!} \sum_{0 \le r \le k} \binom{k}{r} \cdot (N-2r-1)!! \cdot \left( - \frac{1}{\beta N} \right)^{k-r} \\
			&= \frac{1}{(N-1)!!} \E_{\bg \sim \calN\left(0,1\right)} \sum_{0 \le r \le k} \binom{k}{r} \cdot \bg^{N - 2r} \cdot \left( - \frac{1}{\beta N} \right)^{k-r} \\
			&= \frac{1}{(N-1)!!} \E_{\bg \sim \calN\left(0,1\right)} \bg^{N-2k} \left( 1 - \frac{\bg^2}{\beta N} \right)^{k} \\
			&= \frac{1}{(N-1)!!} \int_{-\infty}^{\infty} \frac{1}{\sqrt{2\pi}} e^{-x^2/2} x^{N-2k}\left(1 - \frac{x^2}{\beta N}\right)^{k} \dif x \\
			&= \frac{N^{\frac{N+1}{2}-k}}{2\sqrt{2\pi}(N-1)!!} \int_0^\infty e^{-Nt/2} t^{\frac{N-1}{2}-k} \left( 1 - \frac{t}{\beta} \right)^{k} \dif t\mcom	\numberthis \label{eq:exp-to-integral}
	\end{align*}
	where in the final step we substituted $x = \sqrt{Nt}$.
	To deal with the integral expression \eqref{eq:exp-to-integral}, we will use Laplace's method \Cref{lem:laplace-method}.
	Set $\alpha = \frac{2k}{N} \in [0,\frac{1}{2}]$, and consider the integral
	\[
		\int_{0}^{\infty} e^{N\left( - \frac{t}{2} + \left( \frac{1-\alpha}{2} \right) \log t + \frac{\alpha}{2} \log \left| 1 - \frac{t}{\beta} \right| \right)} \mcom	\numberthis \label{eq:proxy-integral}
	\]
	obtained by ignoring a factor of $t$ in the integral of interest.
	It turns out that $\eqref{eq:exp-to-integral}\le O(1)\cdot \eqref{eq:proxy-integral}$, and so it suffices to bound \eqref{eq:proxy-integral}.
	Indeed, this will follow from the fact that the maximizer $t^{\star}$ of the function of $t$ in the exponent is achieved at a value that is $\Omega(1)$ by using an appropriate version of \Cref{lem:laplace-method} where we bring some factors of $t^\star$ out as another function $f$, and allowing $g$ to depend (very slightly) on $N$---we omit the details, and refer the curious reader to the usual proof of the legitimacy of Laplace's method from, e.g., \cite{But07}.

	Let us find the maximizers of the exponent of this function, given by $g(t) = - \frac{t}{2} + \left( \frac{1-\alpha}{2} \right) \log t + \frac{\alpha}{2} \log \left| 1-\frac{t}{\beta} \right|$. Setting the derivative to $0$ for $t < \beta$ gives
	\begin{align*}
		- \frac{1}{2} + \left(\frac{1-\alpha}{2} \right) \cdot \frac{1}{t} - \frac{\alpha}{2} \cdot \frac{1}{\beta-t} &= 0 \\
		(1-\alpha) \cdot \frac{1}{t} - \alpha \cdot \frac{1}{\beta-t} &= 1 \\
		(1-\alpha) \cdot (\beta-t) - \alpha t &= \beta t-t^2 \\
		t^2 - (1+\beta)t + \beta(1-\alpha) &= 0 \mper
	\end{align*}
	This yields the local maximizer $t^\star = \frac{\beta+1}{2} - \sqrt{\frac{(\beta+1)^2}{4} - \beta(1-\alpha)}$. Observe that because $\alpha < \frac{1}{2}$, $t^\star = \Omega(1)$. Setting the derivative to $0$ for $t > \beta$ gives
	\begin{align*}
		- \frac{1}{2} + \left(\frac{1-\alpha}{2} \right) \cdot \frac{1}{t} + \frac{\alpha}{2} \cdot \frac{1}{\beta-t} &= 0 \\
		(1-\alpha) \cdot \frac{1}{t} + \alpha \cdot \frac{1}{\beta-t} &= 1 \\
		(1-\alpha) \cdot (\beta-t) + \alpha t &= \beta t-t^2 \\
		t^2 + (\beta + 1 - 2\alpha)t + \beta(1-\alpha) &= 0 \mper
	\end{align*}
	It is not difficult to see that this has no positive roots. Therefore, the original function is decreasing on $(\beta,\infty)$, and $t^{\star}$ is in fact the global maximizer. Therefore, Laplace's method \Cref{lem:laplace-method} yields that
	\[ \int_{0}^{\infty} e^{Ng(t)} \dif t = e^{Ng(t^\star)} \cdot \sqrt{\frac{2\pi}{-Ng''(t^\star)}} \left( 1 + O\left( \frac{1}{N} \right) \right) \mper \]
	Indeed, the third and fourth derivatives of $g$ at $t^\star$ are bounded due to the upper bound on $\alpha$. As a result,
	\begin{align*}
		\E \prod_{e \in S} \left( \bone_{e \in \calG} - \frac{1}{\beta N} \right) &\lesssim \frac{N^{\frac{N}{2} - k}}{(N-1)!!} \cdot \frac{e^{Ng(t^\star)}}{\sqrt{-g''(t^\star)}} \\
			&= \frac{N^{\frac{N}{2} - k}}{(N-1)!!} \cdot \frac{e^{-Nt^\star/2}}{\sqrt{-g''(t^\star)}} \cdot \left(t^\star\right)^{N(1-\alpha)/2} \cdot \left( 1 - \frac{t^\star}{\beta} \right)^{N\alpha/2} \mper
	\end{align*}
	We have
	\begin{align*}
		1 - \frac{t^\star}{\beta} &= \frac{1}{\beta} \left( \frac{\beta-1}{2} + \sqrt{\frac{(\beta-1)^2}{4} + \alpha\beta} \right) \\
			&\le \frac{1}{\beta} \left( \beta-1 + \sqrt{\alpha\beta} \right) = \frac{\beta-1}{\beta} + \sqrt{ \frac{\alpha}{\beta} } \mper
	\end{align*}
	Using $\Cref{eq:exp-to-integral}\le O(1)\cdot\Cref{eq:proxy-integral}$
	and plugging this into the above gives
	\[ \E \prod_{e \in S} \left( \bone_{e \in \calG} - \frac{1}{\beta N} \right) \lesssim \frac{N^{\frac{N}{2} - k}}{(N-1)!!} \cdot \frac{e^{-Nt^\star/2}}{\sqrt{-g''(t^\star)}} \cdot \left(t^\star\right)^{N(1-\alpha)/2} \cdot \left( \frac{\beta-1}{\beta} + \sqrt{\frac{2k}{\beta N}} \right)^{k} \mper \]
	Stirling's approximation allows us to write
	\[ \E \prod_{e \in S} \left( \bone_{e \in \calG} - \frac{1}{\beta N} \right) \lesssim \frac{(N-2k-1)!!}{(N-1)!!} \cdot {e^{\frac{N(1-\alpha)}{2}}} \cdot \frac{e^{-Nt^\star/2}}{\sqrt{-g''(t^\star)}} \cdot \left(\frac{t^\star}{1-\alpha}\right)^{N(1-\alpha)/2} \cdot \left( \frac{\beta-1}{\beta} + \sqrt{\frac{2k}{\beta N}} \right)^{k} \mper \]
	We now observe that
	\begin{align*}
		-t^\star + (1-\alpha) \left( 1 + \ln \frac{t^\star}{1-\alpha} \right) \le 0 \mper
	\end{align*}
	Indeed, $1 + \ln \frac{t^\star}{1-\alpha} \le \frac{t^\star}{1-\alpha}$. As a result,
	\[ \E \prod_{e \in S} \left( \bone_{e \in \calG} - \frac{1}{\beta N} \right) \lesssim \frac{(N-2k-1)!!}{(N-1)!!} \cdot \frac{1}{\sqrt{-g''(t^\star)}} \cdot \left( \frac{\beta-1}{\beta} + \sqrt{\frac{2k}{\beta N}} \right)^{k} \mper \]
	Finally, let us deal with the second derivative term. As observed already, we have $\beta-t^\star \le \beta-1 + \sqrt{\alpha\beta}$. We have $t^\star \le \frac{\beta+1}{2} \cdot \frac{4\beta(1-\alpha)}{(\beta+1)^2} \lesssim 1-\alpha$. Consequently,
	\[ -2g''(t^\star) = \frac{1-\alpha}{(t^\star)^2} + \frac{\alpha}{(\beta - t^\star)^2} \ge \frac{1-\alpha}{(t^\star)^2} \gtrsim 1 \mper \]
	To conclude, we use Stirling's approximation to bound
	\begin{align*}
		\frac{(N-2k-1)!!}{(N-1)!!} &\lesssim O(1)^k \cdot \frac{(N-2k)^{(N-2k)/2}}{N^{N/2}} \\
			&\le O \left( \frac{1}{N} \right)^k \mper \qedhere
	\end{align*}
\end{proof}

Let us next describe a bound that puts together the above two, and demonstrates the singleton wins when there are other edges that appear with multiplicity greater than $1$.

\begin{lemma}
	\label{lem:remove-boring-edges}
	Let $S$ be an arbitrary subgraph on vertex set $[N]$, and $T \subseteq S$ such that $T$ and $S \setminus T$ are vertex-disjoint, and $T$ forms a matching. Further suppose that $|S| \le \frac{N}{K}$ for some $K \ge 2$. Let $m_e \ge 1$ be some constant independent of $N$ for each $e \in S$, with $m_e = 1$ for all $e \in T$. Then, for $\calG_{[N]}$ a uniformly random perfect matching on $[N]$, we have
	\[ \left| \E \prod_{e \in S} \left( \bone_{e \in \calG_{[N]}} - \frac{1}{N} \right)^{m_e} \right| \le O\left(\frac{1}{N}\right)^{|S|} \cdot \left( \frac{1}{K} + \sqrt{ \frac{|T|}{N} } \right)^{|T|} \mper \]
\end{lemma}
\begin{proof}
	Let us expand out the non-$T$ terms in the product. This yields an expression of the form
	\begin{align*}
		\left| \E \prod_{e \in S} \left( \bone_{e \in \calG_{[N]}} - \frac{1}{N} \right)^{m_e} \right| &\le \sum_{R \subseteq S \setminus T} \frac{1}{N^{|S \setminus (T \cup R)|}} \left| \E \left[\prod_{e \in T} \left( \bone_{e \in \calG_{[N]}} - \frac{1}{N} \right)^{m_e} \cdot \bone_{R \subseteq \calG_{[N]}} \right] \right| \mper
	\end{align*}
	Observe now that if $R$ consists of any edges that are not vertex-disjoint, the corresponding term on the right is zero. In case all these edges are vertex-disjoint (so they form a matching), we have $\Pr\left[ R \subseteq \calG_{[N]} \right] = \frac{(N-2|R|-1)!!}{(N-1)!!} = O\left( \frac{1}{N} \right)^{|R|}$. Thus,
	\begin{align*}
		\left| \E \prod_{e \in S} \left( \bone_{e \in \calG_{[N]}} - \frac{1}{N} \right)^{m_e} \right| &\le \sum_{R \subseteq S \setminus T} \frac{1}{N^{|S \setminus (T \cup R)|}} \cdot O\left( \frac{1}{N} \right)^{|R|} \cdot \left| \E \left[\prod_{e \in T} \left( \bone_{e \in \calG_{[N]}} - \frac{1}{N} \right)^{m_e} \mid R \subseteq \calG_{[N]} \right] \right| \\
			&= O\left(\frac{1}{N}\right)^{|S \setminus T|} \cdot \sum_{R \subseteq S \setminus T} \left| \E \left[\prod_{e \in T} \left( \bone_{e \in \calG_{[N] \setminus \bigcup_{e \in R} e}} - \frac{1}{N} \right)^{m_e} \right] \right|
	\intertext{We may now apply \Cref{lem:dreg-laplace} to the summand to continue the chain of inequalities as follows.}
		&\le O\left(\frac{1}{N}\right)^{|S \setminus T|} \cdot \sum_{R \subseteq S \setminus T} O \left( \frac{1}{N \left( 1 - \frac{1}{K} \right)} \right)^{|T|} \cdot \left( \frac{1}{K} + \sqrt{ \frac{|T|}{N} } \right)^{|T|} \\
			&\le O \left( \frac{1}{N} \right)^{|S|} \cdot \left(\frac{1}{K} + \sqrt{ \frac{|T|}{N} } \right)^{|T|} \qedhere
	\end{align*}
\end{proof}

Finally, let us put the pieces together.

\begin{lemma}
	\label{lem:dreg-bound-on-subgraph-expectation}
	Let $H^1,H^2$ be a pair of disjoint undirected graphs on $n$ vertices, where the maximum degree in $H^1$ is $2$, and $m_e \ge 2$ an $n$-independent constant for each edge $e$ in $H^2$. Then, 
	\[ \left| \E \prod_{e \in H^1} \left( \bone_{e \in G} - \frac{d}{n} \right) \prod_{e \in H^2} \left( \bone_{e \in G} - \frac{d}{n} \right)^{m_e} \right| \le O\left(\frac{1}{n}\right)^{e(H^1) + e(H^2)} \cdot d^{\frac{1}{2} e(H^1) + e(H^2)} \mcom \]
	where $G$ is, as usual, a random $d$-regular graph on $n$ vertices, drawn from the configuration model.
\end{lemma}

\begin{proof}[Proof of \Cref{th:ass-moment-ass-dreg-ass}]
	Let us set up some notation to make things easier. Let $H^1, H^2$ be a pair of disjoint undirected graphs on $n$ vertices, where the maximum degree of $H^1$ is $2$, and $2 \le m_e \le 4$ some constant for each $e \in H^2$. Denote $k_1 = e(H^1)$, $k_2 = e(H^2)$, and $k = k_1 + k_2$. Also let $\wt{k}_2 = \sum_{e \in H^2} m_e$. Observe that $k_1 \le 2n$.

	In the terminology of \Cref{def:nbp-simple-notation}, $H^1$ indicates the singleton part of $H$ and $H^2$ indicates the rest of $H$. Our goal is to show that
	\[ \left| \E \left[\prod_{e \in H^1} \bM_{e} \cdot \prod_{e \in H^2} \bM_{e}^{m_e}\right] \right| \le O\left( \frac{1}{n} \right)^{e(H^1) + e(H^2)} \mper \]
	Let us start by moving to the configuration model, so we may apply the bounds we have proved. For an edge $e$, let $u_e,v_e \in [n]$ be the two endpoints of $e$, ordered arbitrarily. Let $\calG$ be a draw from the configuration model resulting in $G$.
	Then,
	\begin{align}
	 	&\E \prod_{e \in H^1} \bM_e \prod_{e \in H^2} \bM_{e}^{m_e} \nonumber\\
	 		&\qquad\qquad= \frac{1}{d^{\frac{1}{2} k_1 + \frac{1}{2} \wt{k}_2}} \E \prod_{e \in H^1} \left( \sum_{(i_{e},j_{e}) \in [d]^2} \bone_{(u_e,i_e)(v_e,j_e) \in \calG} - \frac{1}{dn} \right) \prod_{e \in H^2} \left( \sum_{(i_{e},j_{e}) \in [d]^2} \bone_{(u_e,i_e)(v_e,j_e) \in \calG} - \frac{1}{dn} \right)^{m_e} \nonumber \\
	 		&\qquad\qquad= \frac{1}{d^{\frac{1}{2} k_1 + \frac{1}{2} \wt{k}_2}} \sum_{\substack{\{(i_e,j_e)\}_{e \in H^1} \\ \{(i_e^\ell,j_e^\ell)\}_{e \in H^2}^{1 \le \ell \le m_e}}} \E \prod_{e \in H^1} \left( \bone_{(u_e,i_e)(v_e,j_e) \in \calG} - \frac{1}{dn} \right) \prod_{e \in H^2} \prod_{\ell \le m_e} \left( \bone_{(u_e,i_e^\ell)(v_e,j_e^\ell) \in \calG} - \frac{1}{dn} \right) \nonumber \\
	 		&\qquad\qquad= d^{\frac{3}{2} k_1 + \frac{3}{2} \wt{k}_2} \cdot \E_{\substack{\{(i_e,j_e)\}_{e \in H^1} \sim [d]^2 \\ \{(i_e^\ell,j_e^\ell)\}_{e \in H^2}^{1 \le \ell \le m_e} \sim [d]^2}} \E \prod_{e \in H^1} \left( \bone_{(u_e,i_e)(v_e,j_e) \in \calG} - \frac{1}{dn} \right) \prod_{e \in H^2} \prod_{\ell \le m_e} \left( \bone_{(u_e,i_e^\ell)(v_e,j_e^\ell) \in \calG} - \frac{1}{dn} \right) \label{eq:dreg-messy-expression} \tag{$\dagger$}
	\end{align}

	To understand the above quantity, we split into cases based on $e(H^2)$.

	\medskip

	\noindent {\textbf{Case $e(H^2)\ge n\sqrt{d}$.}}
	First, consider the simpler situation where $e(H^2) \ge n\sqrt{d}$. In this case, a simpler argument will suffice where we only win a $\frac{d}{n}$ for each $e \in H^1 \cup H^2$, ignoring the extra wins we get from the singleton edges in $H^1$. Consider the multiplicity of each edge $e \in H^2$, that is, the number of distinct $(i_e^\ell,j_e^\ell)$. It is equal to $1 \le r_e \le m_e$ with probability $O\left( \frac{d^{2r_e}}{d^{2m_e}} \right)$. Further observe that this multiplicity is independent for each edge. Therefore, in this scenario, we may use \Cref{lem:trivial-dreg-bound} (using the aforementioned tail bound on the number of distinct edges) to bound this as
	\begin{align*}
		\eqref{eq:dreg-messy-expression}
		&\le O(1)^k \cdot d^{\frac{3}{2} e(H^1) + \frac{3}{2} \wt{k}_2} \cdot \sum_{(r_e)_{e \in H^2}} \left(\prod_{e \in H^2} \frac{d^{2r_e}}{d^{2m_e}}\right) \cdot \left( \frac{1}{dn} \right)^{e(H^1)} \left(\frac{1}{dn}\right)^{\sum r_e} \\
		&= O(1)^k \cdot \left(\frac{\sqrt{d}}{n}\right)^{e(H^1)} \cdot \frac{1}{d^{\frac{1}{2} \wt{k}_2}} \cdot \sum_{(r_e)_{e \in H^2}} \left(\frac{d}{n}\right)^{\sum r_e} \\
		&\le O(1)^k \cdot \left( \frac{\sqrt{d}}{n} \right)^{k_1} \cdot \frac{1}{d^{k_2}} \cdot \left( \frac{d}{n} \right)^{k_2} = O\left( \frac{1}{n} \right)^{k} \cdot d^{\frac{1}{2} k_1} \mper
	\end{align*}
	To conclude here, note that $d^{\frac{1}{2} k_1} \le O(1)^{n \log d} \le O(1)^{k}$, since $k_2 \ge n\sqrt{d}$.

	\medskip

	\noindent {\textbf{Case $e(H^2) \le n\sqrt{d}$.}}
	Let us thus return to \eqref{eq:dreg-messy-expression}, assuming now that $e(H^2) \le n\sqrt{d}$. We must introduce some more notation for the rest of this proof. For a specific choice of the labels $(i_e,j_e)_{e \in K}$ for some subgraph $K$ on $[n]$, let us refer to the corresponding subgraph in $[n] \times [d]$ as the ``label-extension'' of $K$. Let $H^1_{\ab}$ be the set of all edges $e$ in $H^1$ such that some endpoint of $e$ has degree at least $\sqrt{d}$ in $H^2$. Let $H^1_{\nor}$ be $H^1 \setminus H^1_{\ab}$.
	We will be interested in the value of the following quantity: what is the size of the largest subgraph of $H^1_{\nor}$ that is
	\begin{enumerate}[label=(\alph*)]
		\item a matching, and
		\item isolated from the rest of the label-extension of $H^2$ and $H^1_{\nor}$.
	\end{enumerate}
	For a given choice of labels, let this subgraph be $H^1_*$. Observe that because the endpoints of every edge $e$ in $H^1_{\nor}$ are incident on at most $O(\sqrt{d})$ edges in $H^1 \cup H^2$, the label-extension of $e$ is in $H^1_*$ with probability at least $1 - O\left(\frac{1}{\sqrt{d}}\right)$---it must simply avoid the $O(\sqrt{d})$ edges arising from (the label-extension of) $\left(H^1 \cup H^2\right) \setminus \{e\}$. In particular, $e(H^1_*) = e(H^1_{\nor}) - s$ with probability at most $O\left( \frac{1}{d} \right)^{s/2}$. As in the previous argument, let $r_e$ be the multiplicity of the edge $e \in H^2$ after label-extending. We can then use \Cref{lem:remove-boring-edges} to remove the $\sum r_e + s$ edges arising from $(H^1 \cup H^2) \setminus H^1_*$. Observe that by our assumption that $k = O(n\sqrt{d})$, the constant $K$ in \Cref{lem:remove-boring-edges} can be taken to be $\Omega(\sqrt{d})$. Consequently, recalling again that $e(H^1_{\nor}) \le e(H^1) \le 2n$,
	\begin{align*}
		\eqref{eq:dreg-messy-expression} &\le O(1)^k \cdot d^{\frac{3}{2} k_1 + \frac{3}{2} \wt{k}_2} \cdot \sum_{\substack{(r_e)_{e \in H^2} \ge 1 \\ 0 \le s \le k}} \underbrace{\left(\prod \frac{d^{2r_e}}{d^{2m_e}}\right)}_{\text{probability of $(r_e)$}} \cdot \underbrace{\left( \frac{1}{d} \right)^{s/2}}_{\text{probability of $s$}} \\
		&\qquad\qquad\qquad\cdot \underbrace{\left(\frac{1}{dn}\right)^{\sum r_e + e(H^1_{\ab}) + s} \cdot \left(\frac{1}{dn}\right)^{e(H^1_{\nor}) - s} \cdot \left( \frac{1}{\sqrt{d}} + \sqrt{\frac{e(H^1_{\nor}) - s}{dn}} \right)^{e(H^1_{\nor}) - s}}_{\text{from \Cref{lem:remove-boring-edges}}} \\
		&\le O(1)^k \cdot d^{\frac{3}{2} e(H^1) + \frac{3}{2} \wt{k}_2} \cdot \sum_{\substack{(r_e)_{e \in H^2} \ge 1 \\ 0 \le s \le k}} \left( \prod \frac{d^{2r_e}}{d^{2m_e}} \right) \cdot \left( \frac{1}{d} \right)^{s/2} \cdot \left( \frac{1}{dn} \right)^{\sum r_e + e(H^1_{\ab}) + s} \cdot \left(\frac{1}{dn}\right)^{e(H^1_{\nor}) - s} \cdot \left( \frac{1}{\sqrt{d}} \right)^{e(H^1_{\nor}) - s} \\
		&= O(1)^k \cdot \sum_{\substack{(r_e)_{e \in H^2} \ge 1 \\ 0 \le s \le k}} \left( \frac{1}{n} \right)^{k_1 + \sum r_e} \cdot d^{ \sum r_e - \frac{1}{2} \wt{k}_2 - e(H^1_{\ab}) - \frac{3}{2} e(H^1_{\nor}) + \frac{3}{2} k_1 } \\
		&= O\left( \frac{1}{n} \right)^{k} \cdot d^{ k_2 - \frac{1}{2} \wt{k}_2 + \frac{1}{2} e(H^1_{\ab}) } \le O\left( \frac{1}{n} \right)^{k} \cdot d^{ \frac{1}{2} e(H^1_{\ab}) }
	\end{align*}
	To conclude, we note that much like the simpler setting discussed earlier, $d^{ e(H^1_{\ab}) } = O(1)^{e(H^2)}$, since nearly by definition, $e(H^1_{\ab}) \le O(\sqrt{d}) \cdot e(H^2)$.
	This completes the proof, and brings us to the end of the paper. 
\end{proof}

%% file: main.bbl
\begin{thebibliography}{CGVTvH24}

\bibitem[Ahl79]{Ahl79}
Lars~Valerian Ahlfors.
\newblock {\em Complex analysis}, volume~3.
\newblock McGraw-Hill New York, 1979.

\bibitem[Alo86]{Alo86}
Noga Alon.
\newblock Eigenvalues and expanders.
\newblock {\em Combinatorica}, 6(2):83--96, 1986.

\bibitem[Bas92]{Bas92}
Hyman Bass.
\newblock {The Ihara-Selberg zeta function of a tree lattice}.
\newblock {\em International Journal of Mathematics}, 3(06):717--797, 1992.

\bibitem[BCGZ22]{BCG22}
Charles Bordenave, Djalil Chafa{\"i}, and David Garc{\'\i}a-Zelada.
\newblock Convergence of the spectral radius of a random matrix through its
  characteristic polynomial.
\newblock {\em Probability Theory and Related Fields}, pages 1--19, 2022.

\bibitem[BGBK20]{BGBK20}
Florent Benaych-Georges, Charles Bordenave, and Antti Knowles.
\newblock Spectral radii of sparse random matrices.
\newblock {\em Annales de l'Institut Henri Poincar{\'e}}, 2020.

\bibitem[BHY19]{BHY19}
Roland Bauerschmidt, Jiaoyang Huang, and Horng-Tzer Yau.
\newblock {Local Kesten--McKay law for random regular graphs}.
\newblock {\em Communications in Mathematical Physics}, 369(2):523--636, 2019.

\bibitem[BLM15]{BLM15}
Charles Bordenave, Marc Lelarge, and Laurent Massouli{\'e}.
\newblock Non-backtracking spectrum of random graphs: community detection and
  non-regular {R}amanujan graphs.
\newblock In {\em 2015 IEEE 56th Annual Symposium on Foundations of Computer
  Science}, pages 1347--1357. IEEE, 2015.

\bibitem[BLR25]{BLR25}
Ferenc Bencs, Kuikui Liu, and Guus Regts.
\newblock On zeros and algorithms for disordered systems: mean-field spin
  glasses.
\newblock {\em arXiv preprint arXiv:2507.15616}, 2025.

\bibitem[Bor19]{Bor19}
Charles Bordenave.
\newblock A new proof of {F}riedman's second eigenvalue theorem and its
  extension to random lifts.
\newblock Technical Report 1502.04482v4, arXiv, 2019.
\newblock To appear in Annales scientifiques de l'\'{E}cole normale
  sup\'{e}rieure.

\bibitem[But07]{But07}
Ronald~W Butler.
\newblock {\em Saddlepoint approximations with applications}, volume~22.
\newblock Cambridge University Press, 2007.

\bibitem[CGVTvH24]{CGVTvH24}
Chi-Fang Chen, Jorge Garza-Vargas, Joel~A Tropp, and Ramon van Handel.
\newblock A new approach to strong convergence.
\newblock {\em Annals of Mathematics}, 2024.

\bibitem[EKYY13]{ELKYY13}
L{\'a}szl{\'o} Erd{\H{o}}s, Antti Knowles, Horng-Tzer Yau, and Jun Yin.
\newblock Delocalization and diffusion profile for random band matrices.
\newblock {\em Communications in Mathematical Physics}, 323(1):367--416, 2013.

\bibitem[FK81]{FK81}
Zolt{\'a}n F{\"u}redi and J{\'a}nos Koml{\'o}s.
\newblock The eigenvalues of random symmetric matrices.
\newblock {\em Combinatorica}, 1(3):233--241, 1981.

\bibitem[FM17]{FM17}
Zhou Fan and Andrea Montanari.
\newblock How well do local algorithms solve semidefinite programs?
\newblock In {\em Proceedings of the 49th Annual ACM SIGACT Symposium on Theory
  of Computing}, pages 604--614, 2017.

\bibitem[Fri08]{Fri08}
Joel Friedman.
\newblock {\em A proof of {A}lon's second eigenvalue conjecture and related
  problems}.
\newblock American Mathematical Soc., 2008.

\bibitem[Gir85]{Gir85}
Vyacheslav~L Girko.
\newblock Circular law.
\newblock {\em Theory of Probability \& Its Applications}, 29(4):694--706,
  1985.

\bibitem[Gir18]{Gir18}
Vyacheslav~L Girko.
\newblock {From the first rigorous proof of the circular law in 1984 to the
  circular law for block random matrices under the generalized Lindeberg
  condition}.
\newblock {\em Random Operators and Stochastic Equations}, 26(2):89--116, 2018.

\bibitem[HMY24]{HMY24}
Jiaoyang Huang, Theo McKenzie, and Horng-Tzer Yau.
\newblock {Ramanujan Property and Edge Universality of Random Regular Graphs}.
\newblock {\em arXiv preprint arXiv:2412.20263}, 2024.

\bibitem[HY24]{HY24}
Jiaoyang Huang and Horng-Tzer Yau.
\newblock {Spectrum of Random $d$-regular Graphs Up to the Edge}.
\newblock {\em Communications on Pure and Applied Mathematics},
  77(3):1635--1723, 2024.

\bibitem[Iha66]{Iha96}
Yasutaka Ihara.
\newblock {On discrete subgroups of the two by two projective linear group over
  $p$-adic fields}.
\newblock {\em Journal of the Mathematical Society of Japan}, 18(3):219--235,
  1966.

\bibitem[Nil91]{Nil91}
Alon Nilli.
\newblock On the second eigenvalue of a graph.
\newblock {\em Discrete Mathematics}, 91(2):207--210, 1991.

\bibitem[Tal]{Tal14}
Michel Talagrand.
\newblock Upper and lower bounds for stochastic processes.

\bibitem[Tro15]{Tro15}
Joel~A Tropp.
\newblock An introduction to matrix concentration inequalities.
\newblock {\em Foundations and Trends{\textregistered} in Machine Learning},
  8(1-2):1--230, 2015.

\bibitem[TVK10]{TVK10}
Terence Tao, Van Vu, and Manjunath Krishnapur.
\newblock {Random matrices: Universality of ESDs and the circular law}.
\newblock 2010.

\end{thebibliography}
